\documentclass{article}
\usepackage[utf8]{inputenc}

\title{Fast rates for empirical risk minimization over càdlàg functions with bounded sectional variation norm}
\author{Aurélien F. Bibaut and Mark J. van der Laan}
\date{\today}

\usepackage{mathtools}
\mathtoolsset{showonlyrefs,showmanualtags}
\usepackage{amssymb,amsmath,amsthm}
\usepackage{color}
\usepackage{algorithm}
\usepackage{dsfont}
\usepackage{algpseudocode}
\usepackage{mathtools} 
\usepackage{fullpage}
\usepackage{natbib}
\usepackage{bm}
\newtheorem{theorem}{Theorem}
\newtheorem{theorem*}{Theorem}
\newtheorem{corrolary}{Corrolary}
\newtheorem{assumption}{Assumption}
\newtheorem{definition}{Definition}
\newtheorem{lemma}{Lemma}
\newtheorem{proposition}{Proposition}

\begin{document}

\maketitle

\begin{abstract}
    Empirical risk minimization over classes functions that are bounded for some version of the variation norm has a long history, starting with Total Variation Denoising \citep{Rudin_et_al_1992}, and has been considered by several recent articles, in particular \cite{Fang-Guntuboyina-Sen-2019} and \cite{vdL_2015}. In this article, we consider empirical risk minimization over the class $\mathcal{F}_d$ of càdlàg functions over $[0,1]^d$ with bounded sectional variation norm (also called Hardy-Krause variation).
    
    We show how a certain representation of functions in $\mathcal{F}_d$ allows to bound the bracketing entropy of sieves of $\mathcal{F}_d$, and therefore derive rates of convergence in nonparametric function estimation. Specifically, for sieves whose growth is controlled by some rate $a_n$, we show that the empirical risk minimizer has rate of convergence $O_P(n^{-1/3} (\log n)^{2(d-1)/3} a_n)$. Remarkably, the dimension only affects the rate in $n$ through the logarithmic factor, making this method especially appropriate for high dimensional problems. 
    
    In particular, we show that in the case of nonparametric regression over sieves of càdlàg functions with bounded sectional variation norm, this upper bound on the rate of convergence holds for least-squares estimators, under the random design, sub-exponential errors setting.
\end{abstract}

\section{Introduction}\label{section:introduction}

\paragraph{Empirical risk minimization setting.} We consider the empirical risk minimization setting over classes of real-valued, $d$-variate functions. Suppose that $O_1,...,O_n$ are i.i.d. random vectors with common marginal distribution $P_0$, and taking values in a set $\Theta$. Suppose that $\mathcal{O} \subseteq [0,1]^d \times \mathcal{Y}$, for some integer $d \geq 1$ and some set $\mathcal{Y} \subseteq \mathbb{R}$. Suppose that for all $i$, $O_i = (X_i, Y_i)$, where $X_i \in [0,1]^d$, $Y_i \in \mathcal{Y}$. We suppose that $P_0$ lies in a set of probability distributions over $\mathcal{O}$ that we denote $\mathcal{M}$, and which we call the statistical model. Consider a mapping $\theta$ from the statistical model to a set $\Theta$ of real-valued functions with domain $[0,1]^d$. We call $\Theta$ the parameter set. We want to estimate a parameter $\theta_0$ of the data-generating distribution $P_0$ defined by $\theta_0 = \theta(P_0)$. Let $L:\Theta \rightarrow \mathbb{R}^{\mathcal{O}}$ be a loss mapping, that is for every $\theta \in \Theta$, $L(\theta): \mathcal{O} \rightarrow \mathbb{R}$ is a loss function corresponding to parameter value $\theta$. We suppose that $L$ is a valid loss mapping for $\theta_0$ in the sense that
\begin{align}
\theta_0 = \arg \min_{\theta \in \Theta} P_0 L(\theta).
\end{align}

\paragraph{Statistical model, sieve, and estimator} We define our statistical model implicitly by making a functional class assumption on the parameter set $\Theta$. Specifically, we suppose that $\Theta$ is a subset of the class $\mathcal{F}_d$ of càdlàg functions over $[0,1]^d$ with bounded sectional variation norm \citep{Gill_et_al_1995}. We define now the notion of sectional variation norm. Denote $\mathbb{D}([0,1]^d)$ the set of real-value càdlàg functions with domain $[0,1]^d$. Consider a function $f \in \mathbb{D}([0,1]^d)$.
For all subset $\emptyset \neq s \subseteq [d]$ and for all vector $x \in [0,1]^d$, define the vectors $x_s = (x_j: j \in s)$, $x_{-s} = (x_j: j \notin s)$, and the section $f_s$ of $f$ as the mapping  $f_s(x_s) : x_s \mapsto f(x_s, 0_{-s})$. The sectional variation norm of $f$ is defined as
\begin{align}
    \|f\|_v \equiv |f(0)| + \sum_{\emptyset \neq s \subseteq [d]} \int |f_s(dx_{s})|,
\end{align}
where $[d]$ is a shorthand notation for $\{1,...,d\}$ and  $f_s(dx_{s})$ is the signed measure generated by the càdlàg function $f_s$. Consider a sequence $(\Theta_n)_{n \geq 1}$ of subsets of $\Theta$ such that is non-decreasing for the inclusion. For any $n \geq 1$, we define our estimator $\hat{\theta}_n$ as the empirical risk minimizer over $\Theta_n$, that is
\begin{align}
\hat{\theta}_n = \arg \min_{\theta \in \Theta_n} P_n L(\theta).
\end{align}

\paragraph{Rate of convergence results.} Our main theoretical result states that the empirical risk minimizer $\hat{\theta}_n$ converges to $\theta_0$ at least as fast as $O_P(n^{-1/3} (\log n)^{2(d-1)/3} a_n )$, where $a_n$ depends on the rate of growth of $\Theta_n$ in terms of variation norm. The key to proving this result is a characterization of the bracketing entropy of the class of càdlàg functions with bounded sectional variation norm.
A rate of convergence is then derived based on the famed ``peeling'' technique. 

\paragraph{Tractable representation of the estimator.} \cite{Fang-Guntuboyina-Sen-2019} showed that if the parameter space is itself a set of càdlàg functions with bounded sectional variation norm, then the empirical risk minimizer $\hat{\theta}_n$ can be represented as a linear combination of a certain set of basis functions. (The number of basis functions grows with $n$ and is no larger than $(ne/d)^d$). The empirical risk minimization problem then reduces to a LASSO problem.

\paragraph{Related work and contributions.} \cite{vdL_2015} considered empirical risk minimization over sieves of $\mathcal{F}_d$, under the general bounded loss setting, and showed that it achieves a rate of convergence strictly faster than $n^{-1/4}$ in loss-based dissimilarity. \cite{Fang-Guntuboyina-Sen-2019} consider nonparametric least-squares regression with Gaussian errors and a lattice design, over $\mathcal{F}_{d,M}$ for a certain $M > 0$, and show that the least-squares estimator achieves rate of convergence $n^{-1/3} (\log n)^{C(d)}$ for a certain constant $C(d)$.  In this article, we show that a similar rate of convergence $n^{-1/3} (\log n)^{2(d-1)/3}$ can be achieved under the general setting of empirical risk minimization with unimodal Lipschitz losses (defined formally in section \ref{section:rate_of_convergence_Lipschitz_losses}). We show that this setting covers the case of nonparametric least-squares regression with a bounded dependent variable, and logistic regression, under no assumption on the design.
We also consider the nonparametric regression with sub-exponential errors setting, and show that this $n^{-1/3} (\log n)^{2(d-1)/3} a_n$ rate is achieved by the least-squares estimator over a certain sieve of the set of càdlàg functions with bounded sectional variation norm.

\section{Representation and entropy of the càdlàg functions with bounded sectional variation norm}

As recently recalled by \cite{vdL_2015}, \cite{Gill_et_al_1995} showed that any càdlàg function on $[0,1]^d$ with bounded sectional variation norm can be represented as a sum of $(2^d-1)$ signed measures of bounded variation. This readily implies that any such function can be written as a sum of $(2^d-1)$ differences of scaled cumulative distribution functions, as formally stated in the following proposition.

\begin{proposition}\label{proposition:representation}
Consider  $f \in \mathbb{D}([0,1]^d)$ such that $\|f\|_v \leq M$, for some $M \geq 0$. For all subset $s \subseteq [d]$, and for all vector $x \in [0,1]^d$, define the vector $x_s = (x_j: j \in s)$.
The function $f$ can be represented as follows: for all $x \in [0, 1]^d,$
\begin{align}
    f(x) = f(0) + (M - |f(0)|) \sum_{\emptyset \neq s \subseteq [d] } \int_0^{x_s} \alpha_{s,1} g_{s,1}(dx_s) - \alpha_{s,2} g_{s,2}(dx_s),
\end{align}
where $g_{s,1}$ and $g_{s,2}$ are cumulative distribution functions on the hypercube $[0_s, 1_s]$, and  $\bm{\alpha} = (\alpha_{s, i}: \emptyset \neq s \subseteq [d], i = 1,2 \} \in \Delta^{2^{d+1}-2}$, where $\Delta^{2^{d+1}-2}$ is the $(2^{d+1}-2)$-standard simplex.
\end{proposition}

This and a recent result \citep{Gao2013} on the bracketing entropy of distribution functions implies that the class $\mathcal{F}_{d,M}$ of càdlàg functions over $[0,1]^d$ with variation norm bounded by $M$ has well-controlled entropy, as formalized by the following proposition.

\begin{proposition}\label{proposition-bracketing_entropy}
Let $d\geq 2$ and $M>0$. Denote $\mathcal{F}_{d,M}$ the class of càdlàg functions on $[0,1]^d$ with sectional variation norm smaller than $M$. Suppose that $P_0$ is such that, for all $1 \leq r < \infty$, for all real-valued function $f$ on $[0,1]^d$, $\|f\|_{P_0,r} \leq c(r) \|f\|_{\mu,r}$, for some $c(r) > 0$, and where $\mu$ is the Lebesgue measure. Then for all $1 \leq r < \infty$ and all $0 < \epsilon < 1$, the bracketing  entropy of $\mathcal{F}_{d,M}$ with respect to the $\|\cdot\|_{P_0,r}$ norm satisfies,
\begin{align}
    \log N_{[]}(\epsilon, \mathcal{F}_{d,M}, \|\cdot\|_{P_0,r}) \lesssim C(r,d) M  \epsilon^{-1} |\log(\epsilon/M)|^{2(d-1)},
\end{align}
where $C(r,d)$ is a constant that depends only on $r$ and $d$.
This implies the following bound on the bracketing entropy integral of $\mathcal{F}_{d,M}$ with respect to the $\|\cdot\|_{P_0,r}$ norm: for all $0 < \delta < 1$,
\begin{align}
    J_{[]}(\delta, \mathcal{F}_{d,M}, \|\cdot\|_{P_0,r}) \lesssim \sqrt{C(r,d)} \sqrt{M} \delta^{1/2} |\log(\delta/M)|^{d-1}.
\end{align}
\end{proposition}

\section{Rate of convergence under unimodal Lipschitz losses}\label{section:rate_of_convergence_Lipschitz_losses}

In this section, we present an upper bound on the rate of convergence of $\hat{\theta}_n$ under a general class of loss functions. Essentially, we require the loss to be unimodal and Lipschitz with respect to the parameter, in a pointwise sense. We formally state below the assumptions of our result. Let $(a_n)_{n \geq 1}$ be a non-decreasing sequence of positive numbers, that can potentially diverge to $\infty$.

\begin{assumption}[Control of the variation norm of the sieve]\label{assumption:variation_norm_of_sieve}
Suppose that for all $n \geq 1$,
\begin{align}
\Theta_n \subseteq \{ \theta \in \mathbb{D}([0,1]^d : \|\theta\|_v \leq a_n \}.
\end{align}
\end{assumption}

\begin{assumption}[Loss class]\label{assumption:loss_class}
There exists some $\tilde{L} : \mathbb{R} \times \mathcal{Y} \rightarrow \mathbb{R}$ such that, for any $n$, for any $\theta \in \Theta_n$, and for any $o = (x,y) \in [0,1]^d \times \mathcal{Y}$,
\begin{align}
L(\theta)(x,y) = \tilde{L}(\theta(x), y).
\end{align}
Further assume that $\tilde{L}$ is such that, for any $y$, there is an $u_y$ such that $u \mapsto \tilde{L}(u, y)$ is
\begin{itemize}
\item non-increasing on $(-\infty, u_y]$, and non-decreasing on $[u_y, \infty)$,
\item $a_n$-Lipschitz.
\end{itemize}
\end{assumption}

We will express the rate of convergence in terms of loss-based dissimilarity, which we define now.

\begin{definition}[Loss-based dissimilarity]\label{definition:loss-based_dissimilarity}
Let $n \geq 1$. Denote $\theta_n = \arg\min_{\theta \in \Theta_n} P_0 L(\theta)$. For all $\theta \in \Theta_n$, we define the square of the loss-based dissimilarity $d(\theta, \theta)$ between $\theta$ and $\theta_n$ as the discrepancy 
\begin{align}
    d^2(\theta, \theta_n) = P_0 L(\theta) - P_0 L(\theta_n).
\end{align}
\end{definition}

The third main assumption of our theorem requires the loss $L$ to be smooth with respect to the loss-based dissimilarity.

\begin{assumption}[Smoothness]\label{assumption:smoothness} For every $n$, it holds that
\begin{align}
  \sup_{\theta \in \Theta_n}  \|L(\theta) - L(\theta_n) \|_{P_0,2} \leq a_n d(\theta, \theta_n).
\end{align}
\end{assumption}

We can now state our theorem.

\begin{theorem}\label{theorem:Lipschitz_loss_cvgence_rate}
Consider $\Theta_n$ a sieve such that assumptions \ref{assumption:variation_norm_of_sieve}-\ref{assumption:smoothness} hold for the sequence $a_n$ considered here. Suppose that $a_n = O(n^p)$ for some $p > 0$. Consider our estimator $\hat{\theta}_n$, which, we recall, is defined as the empirical risk minimizer over $\Theta_n$, that is
\begin{align}
    \hat{\theta}_n = \arg \min_{\theta \in \Theta_n} P_n L(\theta).
\end{align}
Suppose that 
\begin{align}
\theta_0 \in \Theta_\infty \equiv \{ \theta \in \Theta : \exists n_0 \text { such that } \forall n \geq n_0,\ \theta \in \Theta_n \}.
\end{align}
Then, we have the following upper bound on the rate of convergence of $\hat{\theta}_n$ to $\theta_0$:
\begin{align}
d(\hat{\theta}_n, \theta_0) = O_P (  a_n  n^{-1/3} (\log n)^{2(d-1) /3}).
\end{align}
\end{theorem}

The reason why we consider a growing sieve $\Theta_n$ is to ensure we don't have to know in advance an upper bound on the variation norm of the losses. The rate $a_n$ impacts the asymptotic rate of convergence and finite sample performance. As the theorem makes clear, the slower we pick $a_n$, the better the speed of convergence. However, for too slow $a_n$, $\theta_0$ might not be included in $\Theta_n$ even for reasonable sample sizes. Note that, if there are reasons to believe that $\|\theta_0 \|_v \leq A$ for some $A > 0$, one can set $a_n = A$ and then the rate of convergence will be $O_P (n^{-1/3} (\log n)^{2(d-1) /3})$.

\section{Applications of theorem \ref{theorem:Lipschitz_loss_cvgence_rate}}\label{section:applications}

\subsection{Least-squares regression with bounded dependent variable}\label{subsection:ls_with_bounded_y}

Consider $\tilde{a}_n$ a non-decreasing sequence of positive numbers, that can potentially diverge to $\infty$. Let $O_1 = (X_1,Y_1),...,(X_n, Y_n)$ be i.i.d. copies of a random vector $O=(X,Y)$ with distribution $P_0$. Suppose that $X$ takes values in $[0,1]^d$ and $Y$ takes values in $\mathcal{Y}_n = [- \tilde{a}_n, \tilde{a}_n ]$. In the setting of least-squares regression, one wants to estimate the regression function $\theta_0 : x \in [0,1]^d \mapsto E_{P_0}[Y|X=x]$ using the square loss $L$ defined, for all $\theta \in \Theta$ as $L(\theta) : (x,y) \mapsto (y - \theta(x))^2$. Let $\Theta_n = \{ \theta \in \mathbb{D}([0,1]^d) : \|\theta\|_v \leq \tilde{a}_n \}$. We consider the least-squares estimator $\hat{\theta}_n$ over $\Theta_n$, defined as
\begin{align}
\hat{\theta}_n = \arg \min_{\theta \in \Theta_n } P_n L(\theta).
\end{align}

Proposition \ref{proposition:least-squares_Lipschitz_loss} and proposition \ref{proposition:least-squares_smooth_loss} below justify that assumptions \ref{assumption:loss_class} and \ref{assumption:smoothness} of theorem \ref{theorem:Lipschitz_loss_cvgence_rate} are satisfied.

\begin{proposition}\label{proposition:least-squares_Lipschitz_loss}
Consider the setting of this subsection. We have, for all $n\geq 1$, $\theta \in \Theta_n$, $x \in [0,1]^d$, and $y \in \mathcal{Y}_n$, that
\begin{align}
L(\theta)((x,y)) = \tilde{L}(\theta(x), y) 
\end{align}
where, $\tilde{L}(u, y) = (y - u)^2$ for all $u,y$.

Furthermore, for all $y \in \mathcal{Y}_n$, the mapping $u \mapsto \tilde{L}(u, y)$ is 
\begin{itemize}
\item non-increasing on $(-\infty, y]$ and non-decreasing on $[y, \infty)$,
\item and $4 \tilde{a}_n$-Lipschitz on $\{\theta(x) : \theta \in \Theta_n, x \in [0,1]^d \}$.
\end{itemize}
\end{proposition}

\begin{proposition}\label{proposition:least-squares_smooth_loss}
Consider the setting of this subsection and recall the definition of the loss-based dissimilarity (see definition \ref{definition:loss-based_dissimilarity}). For all $n \geq 1$, $\theta \in \Theta_n$, we have that
\begin{align}
\| L (\theta) - L(\theta)_n \|_{P_0,2} \leq 4 \tilde{a}_n d_n(\theta, \theta_n).
\end{align}
\end{proposition}

\begin{corrolary}
Set $a_n = 4 \tilde{a}_n$. Then,
\begin{align}
\| \theta - \theta_n \|_{P_0,2} = O_P (  a_n  n^{-1/3} (\log n)^{2(d-1) /3}).
\end{align}
\end{corrolary}

\subsection{Logistic regression}

Consider $\tilde{a}_n$ a non-decreasing sequence of positive numbers that can potentially diverge to $\infty$. Let $O_1 = (X_1,Y_1),...,O_n = (X_n,Y_n)$ be i.i.d. copies of a random vector $O=(X,Y)$, where $X$ takes values in $[0,1]^d$ and $Y \in \{0,1\}$. Denote $P_0$ the distribution of $O$. We want to estimate 
\begin{align}
\theta_0 : x \mapsto \log \left( \frac{E_{P_0}[Y|X=x]}{1 - E_{P_0}[Y|X=x]} \right),
\end{align}
the conditional log-odds function. Let $L$ be the negative log likelihood loss, that is, for all $\theta \in \Theta$, $x \in [0,1]^d$, $y \in \{0,1\}$, $L(\theta)(x,y) = y \log (1 + \exp(-\theta(x))) + (1-y) \log (1 + \exp(\theta(x)))$. Denote $\Theta_n = \{ \theta \in \mathbb{D}([0,1]^d) : \| \theta \|_v \leq \tilde{a}_n \}$. We denote $\hat{\theta}_n$ the empirical risk minimizer over $\Theta_n$, that is
\begin{align}
\hat{\theta}_n = \arg \min_{\theta \in \Theta_n} P_0 L(\theta).
\end{align}
Propositions \ref{proposition:logistic_regression-unimodal-Lipschitz_loss} and \ref{proposition:logistic_regression-smooth_loss} below justify that assumptions \ref{assumption:loss_class} and \ref{assumption:smoothness} of theorem \ref{theorem:Lipschitz_loss_cvgence_rate} are satisfied.

\begin{proposition}\label{proposition:logistic_regression-unimodal-Lipschitz_loss}
Consider the setting of this subsection. We have, for all $n \geq 1$, $\theta \in \Theta_n$, $o=(x,y) \in [0,1]^d \times \{0,1\}$, that
\begin{align}
L(\theta)(o) = \tilde{L}(\theta(x), y),
\end{align}
where, for all $u, y,$
\begin{align}
\tilde{L}(u, y) = y \log (1 + e^{-u}) + (1-y) \log(1 + e^u).
\end{align}
Furthermore, for all $y \in \{0,1\}$, the mapping $u \mapsto \tilde{L}(u,y)$ is
\begin{itemize}
\item non-increasing on $\mathbb{R}$ if $y = 1$,
\item non-decreasing on $\mathbb{R}$ if $y = 0$,
\item $1$-Lipschitz on $\mathbb{R}$.
\end{itemize}
\end{proposition}

\begin{proposition}\label{proposition:logistic_regression-smooth_loss}
Consider the setting of this subsection, and recall the definition of the loss-based dissimilarity. For all $n \geq 1$, we have that
\begin{align}
\| L(\theta) - L(\theta_n) \|_{P_0,2} \leq 2 (1+e^{\tilde{a}_n})^{1/2} d_n(\theta, \theta_n).
\end{align}
\end{proposition}

\begin{corrolary}
Set $a_n = 2 (1 + e^{\tilde{a}_n})^{1/2}$. Then
\begin{align}
d_n(\theta, \theta_n) = O_P (a_n  n^{-1/3} (\log n)^{2(d-1) /3}),
\end{align}
and
\begin{align}
\|\theta - \theta_n\|_{P_0,2} = O_P (  a_n^2  n^{-1/3} (\log n)^{2(d-1) /3}).
\end{align}
\end{corrolary}

\section{Least-squares regression with sub-exponential errors}

In this section we consider a fairly general nonparametric regression setting, namely least-squares regression over a sieve of càdlàg functions with bounded sectional variation norm, under the assumption that the errors follow a subexponential distribution. Although this situation isn't covered by the hypothesis of theorem \ref{theorem:Lipschitz_loss_cvgence_rate}, our general bounded loss result, it is handled by fairly similar arguments. This is a setting of interest in the literature (see e.g. section 3.4.3.2 of \cite{vdV-Wellner-1996}).

Suppose that we collect observations $(X_1,Y_1),...,(X_n, Y_n)$, which are i.i.d. random variable with common marginal distribution $P_0$. Suppose that for all $i$, $X_i \in \mathcal{X} \equiv [0,1]^d$, $Y_i \in \mathcal{Y} \equiv \mathbb{R}$, and that 
\begin{align}
    Y_i = \theta_0(X_i) + e_i,
\end{align}
where $\theta_0 \in \Theta \equiv \{\theta \in \mathbb{D}([0,1]^d): \|\theta\|_v < \infty \}$, and $e_1,...e_n$ are i.i.d. errors that follow a sub-exponential distribution with parameters $(\alpha, \nu)$. Suppose that for all $i$, $X_i$ and $e_i$ are independent. Let $a_n$ be a not-decreasing sequence of positive numbers that can diverge to $\infty$. Define, for all $n \geq 1$, $\Theta_n = \{\theta \in \Theta : \| \theta\|_v \leq a_n \}$.

The following theorem characterizes the rate of convergence of our least-squares estimators, which we explicitly define in the statement of the theorem.

\begin{theorem}\label{thm-rate-least_squares_with_sub-exponential_errors}
Consider the setting of this section. Suppose that $\theta_0 \in \Theta$. Then,  $\hat{\theta}_n$, the least-squares estimator over $\Theta_n$, formally defined as
\begin{align}
\hat{\theta}_n = \arg \min_{\theta \in \Theta_n} \frac{1}{n} \sum_{i=1}^n (Y_i - \theta(X_i))^2,
\end{align} 
satisfies
\begin{align}
    \| \hat{\theta}_n - \theta_0 \|_{P_0, 2} = O_P(( (\tilde{C}(\alpha, \nu) + 3 ) a_n + \|\theta_0\|_\infty )  n^{-1/3} (\log n)^{2(d-1)/3} ).
\end{align}
where the constant $\tilde{C}(\alpha, \nu)$ is defined in the appendix.
\end{theorem}

\bibliography{biblio}

\pagebreak

\appendix

\section{Proof of the bracketing entropy bound (proposition \ref{proposition-bracketing_entropy})}

The proof of proposition \ref{proposition-bracketing_entropy} relies on the representation of càdlàg functions with bounded sectional variation norm and on the the three results below. For all $d \geq 1$, $M > 0$, denote
\begin{align}
\mathcal{F}_{d,M} = \{ f \in \mathbb{D}([0,1]^d) : \| f\|_v \leq M \}.
\end{align}

The first result characterizes the bracketing entropy of the set of $d$-dimensional cumulative distribution functions.

\begin{lemma}[Theorem 1.1 in \cite{Gao2013}] \label{lemma:bracketing_cdfs}
Let $\mathcal{G}_d$ be the set of probability distributions on $[0,1]^d$. For $1 \leq r < \infty$ and $d \geq 2$,
\begin{align}
\log N_{[]}(\epsilon, \mathcal{G}_d, \|\cdot\|_{\mu,r}) \leq C'(d, r) \epsilon^{-1} | \log \epsilon |^{2(d-1)}
\end{align} 
for some constant $C'(d,r)$ that only depends on $d$ and $r$, and where $\mu$ is the Lebesgue measure on $[0,1]^d$.
\end{lemma}

As an immediate corrolary, the following result holds for bracketing numbers w.r.t. $\|\cdot\|_{P_0,2}$.
\begin{corrolary}\label{corrolary:bracketing_cdfs}
Let $1 \leq r < \infty$. Suppose there exists a constant $c(r)$ such that $\|\cdot\|_{P_0,r} \leq c(r) \|\cdot\|_{\mu,r}$. Then,
\begin{align}
\log N_{[]}(\epsilon, \mathcal{G}_d, \|\cdot\|_{P_0,r}) \leq \tilde{C}(d, r) \epsilon^{-1} | \log \epsilon |^{2(d-1)}
\end{align}
for some constant $\tilde{C}(r,d)$ that only depends on $r$ and $d$.
\end{corrolary}

The next lemma will be useful to bound the bracketing entropy integral.
\begin{lemma}
For any $d \geq 0$ and any $0 < \delta \leq 1$, we have that
\begin{align}\label{lemma-integral}
    \int_0^\delta \epsilon^{-1/2} (\log (1/\epsilon))^{d-1} d \epsilon \lesssim \delta^{1/2} (\log (1/\delta))^{d-1}.
\end{align}
\end{lemma}

\begin{proof}
The result is readily obtained by integration by parts.
\end{proof}

We can now present the proof of proposition \ref{proposition-bracketing_entropy}.

\begin{proof}
We will first upper bound the ($\epsilon, \|\cdot\|_{P_0,r})$-bracketing number for $\mathcal{F}_{d,1}$. An upper bound on the ($\epsilon, \|\cdot\|_{P_0,r})$-bracketing number for $\mathcal{F}_{d,M}$ will then be obtained at the end of the proof by means of change of variable.
Recall that any function in $\mathcal{F}_{d,1}$ can be written as 
\begin{align}
    f = \sum_{ s \subseteq [d] } \alpha_{s, 1} g_{s,1} - \alpha_{s, 2} g_{s, 2},
\end{align}
with $g_{\emptyset, 1} = g_{\emptyset, 2} = 1$, and for all $\emptyset \neq s \subseteq [d]$,  $g_{s,1}, g_{s,2} \in \mathcal{G}_s$, and $\bm{\alpha} = (\alpha_{s, i}:  s \subseteq [d], i = 1,2 ) \in \Delta^{2^{d+1}}$, where $\Delta^{2^{d+1}}$ is the $2^{d+1}$-standard simplex.

Let $\epsilon > 0$. Denote $N(\epsilon / 2^{d+1}, \Delta^{2^{d+1}}, \|\cdot \|_\infty)$ the $(\epsilon / 2^{d+1}, \|\cdot\|_\infty)$-covering number of $\Delta^{2^{d+1}}$. Let 
\begin{align}
    \{ \bm{\alpha}^{(j)} : j=1,...,N(\epsilon / 2^{d+1}, \Delta^{2^{d+1}}, \|\cdot\|_\infty) \}
\end{align} be an $(\epsilon/2^{d+1}, \|\cdot\|_\infty)$-covering of $\Delta^{2^{d+1}}$. For all $s \subseteq [d]$, denote $N_{[]}(\epsilon, \mathcal{G}_s, \|\cdot\|_{P_0,r})$ the $(\epsilon, \|\cdot\|_{P_0,r})$-bracketing number of $\mathcal{G}_s$, and let 
\begin{align}
    \{ (l_s^{(j)}, u_s^{(j)}) : j=1,...,N_{[]}(\epsilon, \mathcal{G}_s, \|\cdot\|_{P_0,r}) \}
\end{align}
be an $(\epsilon, \|\cdot\|_{P, r})$-bracketing of $\mathcal{G}_s$.

\paragraph{Step 1: Construction of a bracket for $\mathcal{F}_{d,1}$.} We now construct a bracket for $f$ from the cover for $\Delta^{2^{d+1}}$ and the bracketings for $\mathcal{G}_s, \emptyset \neq s \subseteq [d]$, we just defined. By definition of an $(\epsilon/2^{d+1}, \|\cdot\|_\infty)$-cover, there exists $j_0 \in \{1,...,N(\epsilon / 2^{d+1}, \Delta^{2^{d+1}}, \|\cdot\|_{P_0,r}) \}$ such that $\| \bm{\alpha} - \bm{\alpha}^{(j_0)} \|_\infty \leq \epsilon / 2^{d+1}$. Consider $s \subseteq [d]$, $i \in \{1,2\}$. By definition of an $(\epsilon, \|\cdot\|_{P_0,r})$-bracket exists $j_{s,i} \in \{1,...,N_{[]}(\epsilon, \mathcal{G}_s, \|\cdot \|_{P,r})$ such that
\begin{align}
    l_s^{ (j_{s,i}) } \leq g_{s,i} \leq u_s^{ (j_{s,i}) }.
\end{align}
This and the fact that 
\begin{align}
    \alpha_{s,i}^{(j_0)} - \epsilon / 2^{d+1} \leq \alpha_{s,i} \leq \alpha_{s, i}^{(j_0)} + \epsilon / 2^{d+1},
\end{align}
will allow us to construct a bracket for $\alpha_{s,i} g_{s,i}$. Some care has to be taken due to the fact $l_s^{j_{s,i}}$ can be negative (as bracketing functions do not necessarily belong to the class they bracket). Observe that, since $\alpha_{s,i} \geq 0$, we have
\begin{align}
    \alpha_{s, i} l_s^{(j_{s,i})} \leq \alpha_{s,i} g_{s,i} \leq \alpha_{s, i} u_s^{(j_{s,i})}.
\end{align}
Denoting $(l_s^{(j_{s,i})})^+$ and $(l_s^{(j_{s,i})})^-$ the positive and negative part of $l_s^{(j_{s,i})}$, we have that
\begin{align}
    (\alpha_{s,i}^{(j_0)} - \epsilon / 2^{d+1}) (l_s^{(j_{s,i})})^+ &\leq \alpha_{s,i} l_s^+\\
    \text{ and } - (\alpha_{s,i}^{(j_0)} + \epsilon / 2^{d+1}) (l_s^{(j_{s,i})})^- &\leq - \alpha_{s,i} l_s^-.
\end{align}
Therefore,
\begin{align}
    \alpha_{s,i}^{(j_0)} l_s^{(j_{s,i})} - \epsilon / 2^{d+1} |l_s^{(j_{s,i})}| \leq \alpha_{s,i} l_s^{(j_{s,i})}.
\end{align}
Since $u_{s,i}^{(j_{s,i})} \geq 0$ (at it is above at least one cumulative distribution function from $\mathcal{G}_s$), and $\alpha_{s,i}^{(j_0)} + \epsilon / 2^{d+1} \geq \alpha_{s,i}$, we have that 
\begin{align}
    \alpha_{s,i} g_{s,i} \leq (\alpha_{s,i}^{(j_0)} + \epsilon / 2^{d+1}) u_s^{(j_{s,i})}.
\end{align}
Therefore, we have shown that
\begin{align}
    \alpha_{s,i}^{(j_0)} l_s^{(j_{s,i})}  - \epsilon / 2^{d+1} |l_s^{(j_{s,i})}| \leq \alpha_{s,i} g_{s,i} \leq (\alpha_{s,i}^{(j_0)} + \epsilon / 2^{d+1}) u_s^{(j_{s,i})}.
\end{align}
Summing over $s \subset \{1, ... ,d\}$ and $i=1,2$, we have that
\begin{align}
    \Lambda_1 - \Gamma_2 \leq f \leq \Gamma_1 - \Lambda_2,
\end{align}
where, for $i=1,2$,
\begin{align}
    \Lambda_i  &= \sum_{s \subseteq [d] } \alpha_{s,i}^{(j_0)} l_s^{j_{s,i}} - \epsilon / 2^{d+1} | l_s^{j_{s,i}}|,\\
    \text{and } \Gamma_i  &= \sum_{s \subseteq [d] } (\alpha_{s,i}^{(j_0)} + \epsilon / 2^{d+1})  u_s^{j_{s,i}}.
\end{align}

\paragraph{Step 2: Bounding the size of the brackets.} For $i=1,2$,
\begin{align}
    0 \leq \Gamma_i - \Lambda_i =  \sum_{s \subseteq [d] } \alpha_{s,i}^{(j_0)} ( u_s^{j_{s,i}} - l_s^{j_{s,i}} ) + \epsilon / 2^{d+1} (u_s^{j_{s,i}} + |l_s^{j_{s,i}}|).
\end{align}
Since, for every $s \subseteq [d]$, $i=1,2$ $u_s^{j_{s,i}}$ and $l_s^{j_{s,i}}$ are at most $\epsilon$-away in $\|\cdot \|_{P,r}$ norm from a cumulative distribution function, we have that $\|u_s^{j_{s,i}}\|_{P,r} \leq 1 + \epsilon$ and $\| l_s^{j_{s,i}} \|_{P,r} \leq 1+\epsilon$. By definition, for all $s \subseteq [d]$, $i=1,2$, $\| u_s^{j_{s,i}} - l_s^{j_{s,i}}\|_{P,r} \leq \epsilon$. Therefore, from the triangle inequality,
\begin{align}
    \| \Gamma_i - \Lambda_i \|_{P,r} &\leq \epsilon \sum_{s \in \subseteq [d] } \alpha_{s,i} + \epsilon (1 + \epsilon).
\end{align}
Therefore, using the triangle inequality one more time,
\begin{align}
    \| \Gamma_1 - \Lambda_2 - (\Lambda_1 - \Gamma_2) \|_{P,r} &\leq \epsilon \sum_{s \subseteq [d] } \alpha_{s,1} + \alpha_{s,2} + 2 \epsilon (1 + \epsilon)\\
    &\leq 3 \epsilon + 2 \epsilon^2.
\end{align}
Since cumulative distribution functions have range $[0,1]$, brackets never need to be of size larger than 1. Therefore, without loss of generality, we can assume that $\epsilon \leq 1$. Therefore, pursuing the above display, we get
\begin{align}
    | \Gamma_1 - \Lambda_2 - (\Lambda_1 - \Gamma_2) \|_{P,r} \leq 5 \epsilon.
\end{align}

\paragraph{Step 3: Counting the brackets.} Consider the set of brackets of the form $(\Gamma_1 - \Lambda_2, \Lambda_1 -\Gamma_2)$, where, for $i=1,2$,
\begin{align}
    \Lambda_i  &= \sum_{s \subseteq [d] } \alpha_{s,i}^{(j_0)} l_s^{j_{s,i}} - \epsilon / 2^{d+1} | l_s^{j_{s,i}}|,\\
    \text{and } \Gamma_i  &= \sum_{s \subseteq [d] } (\alpha_{s,i}^{(j_0)} + \epsilon / 2^{d+1})  u_s^{j_{s,i}},
\end{align}
where $j_0 \in \{ 1 , ... , N(\epsilon / 2^{d+1}), \Delta^{2^{d+1}}, \|\cdot\|_\infty) \}$ and for any $s, i$ $j_{s,i} \in \{1,...,N_{[]}(\epsilon, \mathcal{G}_s, \|\cdot\|_{P_0,r}) \}$. From step 1 and step 2, we know that this set of brackets is a $(5\epsilon, \|\cdot\|_{P_0,r})$-bracketing of $\mathcal{F}_1$. Its cardinality is no larger than the cardinality of its index set. Therefore
\begin{align}
    N_{[]}(5 \epsilon, \mathcal{F}_1, \|\cdot \|_{P,r} ) \leq N(\epsilon / 2^{d+1}, \Delta^{2^{d+1}}, \|\cdot\|_\infty) \prod_{s \subseteq [d]} N_{[]}(\epsilon, \mathcal{G}_s, \|\cdot\|_{P_0,r})^2.
\end{align}
The covering number of the simplex can be bounded (crudely) as follows:
\begin{align}
     N(\epsilon / 2^{d+1}, \Delta^{2^{d+1}}, \|\cdot\|_\infty) \leq \left( \frac{2^{d+1}}{\epsilon} \right)^d
\end{align}
therefore
\begin{align}
  \log N(\epsilon / 2^{d+1}, \Delta^{2^{d+1}}, \|\cdot\|_\infty) \leq d \log (1 / \epsilon) + d(d+1) \log 2.
\end{align}
From corrolary \ref{corrolary:bracketing_cdfs}, 
\begin{align}
    \log N_{[]}(\epsilon, \mathcal{G}_s, \|\cdot\|_{P_0,r}) \leq C(r,d) \epsilon^{-1} |\log \epsilon|^{2(d-1)}.
\end{align}
Therefore, 
\begin{align}
    \log N_{[]}(5 \epsilon, \mathcal{F}_1, \|\cdot\|_{P_0,r}) &\leq \tilde{C}(r,d) 2^{d+2} \epsilon^{-1} |\log \epsilon|^{2(d-1)} + d \log (1 / \epsilon) + d(d+1) \log 2  \\
    & \lesssim \tilde{C}(r,d) 2^{d+2} \epsilon^{-1} |\log \epsilon|^{2(d-1)}.
\end{align}
Therefore, doing a change of variable, (and for a different constant absorbed in the $\lesssim$ symbol),
\begin{align}
     \log N_{[]}(\epsilon, \mathcal{F}_M, \|\cdot\|_{P_0,r}) \lesssim \tilde{C}(r,d) 2^{d+2} M \epsilon^{-1} |\log (\epsilon/M)|^{2(d-1)}.
\end{align}
The wished claims hold for $C(r,d) = 2^{d+2} \tilde{C}(r,d)$.
\end{proof}

\section{Proofs of theorem \ref{theorem:Lipschitz_loss_cvgence_rate} and preliminary results}

\subsection{Overview and preliminary lemmas}
The proof of the theorem relies on theorem 3.4.1 in \cite{vdV-Wellner-1996}, which gives an upper bound on the rate of convergence of the estimator in terms of the ``modulus of continuity'' of an empirical process indexed by a difference in loss functions. We bound this ``modulus of continuity'' by using a maximal inequality for this empirical process. This maximal inequality is expressed in terms of the bracketing entropy integrals of the class of function $\mathcal{L}_n = \{L(\theta) - L(\theta)_n : \theta \in \Theta_n \}$. We link the bracketing entropy of $\mathcal{L}_n$ to the one of $\Theta_n$ through lemma \ref{lemma:bracketing_preservation}.

We first restate here the theorem 3.4.1. in \cite{vdV-Wellner-1996}.

\begin{theorem}[Theorem 3.4.1 in \cite{vdV-Wellner-1996}]\label{thm_peeling} For each $n$, let $\mathbb{M}_n$ and $M_n$ be stochastic processes indexed by a set $\Theta$. Let $\theta_n \in \Theta$ (possibly random) and $0 \leq \delta_n \leq \eta$ be arbitrary, and let $\theta \mapsto d_n(\theta, \theta_n)$ be an arbitrary map (possibly random) from $\Theta$ to $[0, \infty)$. Suppose that, for every $n$ and $\delta_n \leq \delta \leq \eta$,
\begin{align}
    \sup_{\substack{\theta \in \Theta_n\\ \delta/2 \leq d_n(\theta, \theta_n) \leq \delta }} M_n(\theta) - M_n(\theta_n) \leq - \delta^2, \label{thm_peeling-population_risk_and_distance_condition}
\end{align} 
\begin{align}
    E^* \sup_{\substack{\theta \in \Theta_n\\ \delta/2 \leq d_n(\theta, \theta_n) \leq \delta }} \sqrt{n}[(\mathbb{M}_n - M_n)(\theta) - (\mathbb{M}_n - M_n)(\theta_n)]^+ \lesssim \phi_n(\delta), \label{thm_peeling-modulus_condition}
\end{align}
for functions $\phi_n$ such that $\delta \mapsto \phi_n(\delta) / \delta^\alpha$ is decreasing on $(\delta_n, \eta)$ for some $\alpha < 2$. Let $r_n \lesssim \delta_n^{-1}$ satisfy
\begin{align}
    r_n^2 \phi_n \left(\frac{1}{r_n} \right) \leq \sqrt{n}, \text{ for every } n. \label{thm_peeling-rate_and_modulus_condition}
\end{align}
If the sequence $\hat{\theta}_n$ takes its values in $\Theta_n$ and satisfies 
\begin{align}
    \mathbb{M}_n(\hat{\theta}_n) \geq \mathbb{M}_n(\theta_n) - O_P(r_n^{-2}) \label{thm_peeling-empirical_maximizer_condition}
\end{align} and $d_n(\theta, \theta_n)$ converges to zero in outer probability, then $r_n d_n(\hat{\theta}_n, \theta_n) = O_P^*(1)$. If the displayed conditions are valid for $\eta=\infty$, then the condition that $\theta_n$ is consistent is unnecessary.
\end{theorem}

The quantity $\phi_n(\delta)$ is the so-called ``modulus of continuity'' of the centered process $\sqrt{n}(\mathbb{M}_n - M_n)$ over $\Theta_n = \mathcal{F}_n$. Theorem \ref{thm_peeling} essentially teaches us that the rate of the modulus of continuity gives us the (an upper bound on) the rate of convergence of the estimator.

We now restate the maximal inequality that we will use to bound the modulus of continuity.
\begin{lemma}[Lemma 3.4.2 in \cite{vdV-Wellner-1996}]\label{lemma-maximal_inequality_L2_brackets}
Let $\mathcal{F}$ be a class of measurable functions such that $Pf^2 < \delta^2$ and $\|f\|_\infty \leq M$ for every $f \in \mathcal{F}$. Then
\begin{align}
    E_P^* \sup_{f \in \mathcal{F}} \sqrt{n} |(P_n - P) f| \lesssim J_{[]}(\delta, \mathcal{F}, L_2(P)) \left(1 + \frac{J_{[]}(\delta, \mathcal{F}, L_2(P))}{\delta^2 \sqrt{n}} M \right).
\end{align}
\end{lemma}

Application of the above maximal inequality is what will allow us to bound the ``modulus of continuity''. The following lemma will be useful to upper bound the entropy integral of $\mathcal{L}_n = \{ L(\theta) - L(\theta_n) : \theta \in \Theta_n \}$ in terms of the entropy integral of $\Theta_n$.
\begin{lemma}\label{lemma:bracketing_preservation}
Let $F:\mathbb{R} \times \mathcal{W} \rightarrow \mathbb{R}$ a mapping such that, for any $w \in \mathcal{W}$, there exists $a_w \in \mathbb{R}$ such that $a \mapsto F(a, w)$ is
\begin{itemize}
\item non-increasing on $(-\infty, a_w]$,
\item non-decreasing on $[a_w, \infty)$,
\item $M$-Lipschitz for some $M$ that does not depend on $w$.
\end{itemize}
Let $\mathcal{A}$ a set of real-valued functions defined on a set $\mathcal{V}$. Let
\begin{align}
\mathcal{B} = \{ (v,w) \in \mathcal{V} \times \mathcal{W} \mapsto F(a(v), w) : a \in \mathcal{A} \}.
\end{align}
Let $r \geq 1$, and let $P$ a probability distribution over $\mathcal{V} \times \mathcal{W}$. Then, for any $\delta > 0$,
\begin{align}
N_{[]}(\delta, \mathcal{B}, \|\cdot\|_{P_0,r}) \leq N_{[]}(\delta / M, \mathcal{A}, \|\cdot\|_{P_0,r}),
\end{align}
and 
\begin{align}
J_{[]}(\delta, \mathcal{B}, \|\cdot\|_{P_0,r}) \leq M J_{[]}(\delta / M, \mathcal{A}, \|\cdot\|_{P_0,r}).
\end{align}
(Note that the above quantities might not be finite.)
\end{lemma}
We defer the proof of this lemma to subsection \ref{subsection:proof_of_bracketing_preservation_lemma}.
The following lemma shows that the variation norm dominates the supremum norm.
\begin{lemma}\label{lemma:sup_norm_dominated_by_svn}
For all $f \in \mathbb{D}([0,1]^d)$,
\begin{align}
\|f\|_\infty \leq \|f\|_v.
\end{align}
\end{lemma}
\subsection{Proof of theorem \ref{theorem:Lipschitz_loss_cvgence_rate}}
We now present the proof of theorem \ref{theorem:Lipschitz_loss_cvgence_rate}. 

\begin{proof}[Proof of theorem \ref{theorem:Lipschitz_loss_cvgence_rate}. ] The proof essentially consists of checking the assumptions of theorem \ref{thm_peeling} for a certain choice of $\mathbb{M}_n$, $M_n$, $d_n$ and $r_n$. Specifically, we set, for every $\theta \in \Theta_n$, and every $n$,
\begin{align}
    \mathbb{M}_n(\theta) &= - P_n L(\theta),\\
    M_n(\theta) &= -P_0 L(\theta),\\
    \theta_n &= \arg \min_{\theta \in \Theta_n} P_0 L(\theta),\\
    d_n^2(\theta, \theta_n) &= P_0 L(\theta) - P_0 L(\theta_n),\\
    r_n &= C(r,d)^{-1/3} a_n^{-1} n^{1/3} (\log n)^{-2(d-1)/3}.
\end{align}
Further set $\eta = \infty$ and $\delta_n = 0$. From now, we proceed in three steps.

\paragraph{Step 1: Checking condition \ref{thm_peeling-population_risk_and_distance_condition}.} By definition of $M_n$ and by definition of the loss-based dissimilarity, we directly have, for every $\theta \in \Theta_n$,
\begin{align}
    M_n(\theta) - M_n(\theta_n) = -P_0 (L(\theta) - L(\theta_n))  = - d_n^2(\theta, \theta_n).
\end{align}
Therefore, condition \ref{thm_peeling-population_risk_and_distance_condition} holds.

\paragraph{Step 2: Bounding the modulus of continuity.} We want to bound 
\begin{align}
    &E_{P_0} \sup_{\substack{\theta \in\Theta_n \\ d_n(\theta, \theta_n) \leq \delta}}  |(\mathbb{M}_n - M_n)(\theta) - (\mathbb{M}_n - M_n)(\theta_n)|\\
    &= E_{P_0} \sup_{\substack{\theta \in\Theta_n \\ d_n(\theta, \theta_n) \leq \delta}} |(P_n - P_0) (L(\theta) - L(\theta_n))|\\
    &= E_{P_0} \sup_{g \in \mathcal{G}_n(\delta)} |(P_n - P_0) g|, \label{proof_thm_rate_bounded_loss-modulus_rewritten}
\end{align}
where 
\begin{align}
    \mathcal{G}_n(\delta) = \{L(\theta) - L(\theta_n) : \theta \in\Theta_n, d_n(\theta, \theta_n) \leq \delta \}.
\end{align} 
We now further characterize the set $\mathcal{G}_n(\delta)$.
From assumption \ref{assumption:smoothness}, for all $\theta \in \Theta_n$,
$\| L(\theta) - L(\theta_n) \|_{P_0,2} \leq a_n  d_n(\theta, \theta_n)$. Therefore, denoting $\mathcal{L}_n = \{ L(\theta) - L(\theta_n) : \theta \in \Theta_n \}$ and $\mathcal{L}_n(\delta) = \{ g \in \mathcal{L} : \|g\|_{P_0, 2} \leq \delta \}$, we have that $\mathcal{G}_n(\delta) \subseteq \mathcal{L}_n(a_n \delta)$. We now turn to bounding in supremum norm the class $\mathcal{L}_n$. From assumption \ref{assumption:loss_class}, for all $\theta \in \Theta_n$, $\|L(\theta) - L(\theta_n)\|_\infty \leq a_n \| \theta - \theta_n\|_\infty$. From the definition of $\Theta_n$ and lemma \ref{lemma:sup_norm_dominated_by_svn}, we have that, for all $\theta \in \Theta_n$, $\|\theta - \theta_n\|_\infty \leq 2 a_n$, which implies that $\|L(\theta) - L(\theta_n)\|_\infty \leq 2 a_n^2$.
Therefore, from \eqref{proof_thm_rate_bounded_loss-modulus_rewritten} and the maximal inequality of lemma \ref{lemma-maximal_inequality_L2_brackets}, we have
\begin{align}
    &E_{P_0} \sup_{\substack{\theta \in\Theta_n \\ d_n(\theta, \theta_n) \leq \delta}}  |(\mathbb{M}_n - M_n)(\theta) - (\mathbb{M}_n - M_n)(\theta_n)|\\
    &\leq E_{P_0} \sup_{g \in \mathcal{L}_n(a_n \delta)} |(P_n - P_0) g| \\
    &\leq \frac{\phi_n(\delta)}{\sqrt{n}},
\end{align}
with
\begin{align}
    \phi_n(\delta) \equiv & J_{[]}(a_n \delta, \mathcal{L}_n ,  \|\cdot\|_{P_0,2}) \left(1 + \frac{J_{[]}(a_n \delta, \mathcal{L}_n,  \|\cdot\|_{P_0,2})}{(a_n \delta)^2 \sqrt{n}} 2  a_n^2 \right).
\end{align}

\paragraph{Step 3: Checking the rate condition $r_n^2 \phi_n (1 /r_n) \leq \sqrt{n}$.} From lemma \ref{lemma:bracketing_preservation}, and then from proposition \ref{proposition-bracketing_entropy},
\begin{align}
    J_{[]}(a_n \delta, \mathcal{L}_n , \|\cdot\|_{P_0,2}) \lesssim & a_n J_{[]}(\delta, \Theta_n, L_2(P_0)) \\
    \lesssim & a_n C(r,d)^{1/2} a_n^{1/2} \delta^{1/2} ( \log (a_n / \delta) )^{d-1} \\
    \lesssim & C(r,d)^{1/2} a_n^{3/2} \delta^{1/2} (\log (a_n / \delta))^{d-1}.
\end{align}
Recall that we set
\begin{align}
    r_n = C(r,d)^{-1/3} a_n^{-1} n^{1/3} (\log n)^{-2(d-1)/3}.
\end{align}
Since we supposed that $a_n = O(n^p)$ for some $p > 0$, we have that $\log (a_n r_n) \lesssim \log n$. Therefore,
\begin{align}
    r_n^2 \phi_n(1/r_n) \lesssim & r_n^2 C(r,d)^{1/2} a_n^{3/2} r_n^{-1/2} (\log (a_n r_n))^{d-1} \left( 1 + \frac{C(r,d)^{1/2} a_n^{3/2} r_n^{-1/2} (\log (a_n r_n))^{d-1} }{(a_n / r_n)^2 \sqrt{n}} 2 a_n^2 \right)  \\
    \lesssim & C(r,d)^{1/2} a_n^{3/2} r_n^{3/2} (\log n)^{d-1} \left( 1 + 2 \frac{C(r,d)^{1/2} a_n^{3/2} r_n^{3/2} (\log n)^{d-1} }{\sqrt{n}}  \right) \\
     \lesssim & 3 \sqrt{n}.
\end{align}
\end{proof}

\subsection{Proof of technical lemmas \ref{lemma:bracketing_preservation} and \ref{lemma:sup_norm_dominated_by_svn}}\label{subsection:proof_of_bracketing_preservation_lemma}
\begin{proof}[Proof of lemma \ref{lemma:bracketing_preservation}]
Let $[l,u]$ an $(\epsilon, \|\cdot\|_{P,r})$-bracket for $\mathcal{A}$ and let $a \in \mathcal{A}$ such that $a \in [l,u]$. Define, for all $(v,w) \in \mathcal{V} \times \mathcal{W}$,
\begin{align}
\Lambda(v,w) = \begin{cases} 
F(a_w, w) &\text{ if } l(v) \leq a_w \leq u(v),\\
F(l(v), w) \wedge F(u(v), w) & \text{ otherwise,}
\end{cases}
\end{align}
and 
\begin{align}
\Gamma(v, w) = F(l(v), w) \vee F(u(v), w).
\end{align}
We claim that $(\Lambda, \Gamma)$ is an $(M\epsilon, \|\cdot\|_{P,r})$-bracket for $(u,v) \mapsto F(a(v), w)$. We distinguish three cases. Let $(u,v) \in \mathcal{V} \times \mathcal{W}$.

\paragraph{Case 1.} Suppose that $l(v) \leq a_w \leq u(v)$. Then since $a \mapsto F(a, w)$ reaches its minimum in $a_w$, we have that $\Lambda(v, w) = F(a_w, w) \leq F(a(v), w)$. If $a(v) \in [a_w, u(v)]$, then, as $a \mapsto F(a, w)$ is non-decreasing on $[a_w, \infty)$, we have that $F(a(v), w) \leq F(u(v), w)$. If  $a(v) \in [l(v), a_w]$, then, as $a \mapsto F(a, w)$ is non-increasing on $(-\infty, a_w]$, $F(a(v), w) \leq F(l(v), w)$. Thus
$F(a(v), w) \leq F(l(v), w)\vee F(u(v), w) = \Gamma(v, w)$.

Observe that, under $[a_w \in [l(v), u(v)]$, we have that $|l(v) - a_w| \leq |u(v) - l(v)|$ and $|u(v) - a_w| \leq |u(v) - l(v)|$. Therefore, if $\Gamma(v, w) = F(u(v), w)$,
\begin{align}
|\Gamma(v, w) - \Lambda(v,w)| = |F(u(v), w) - F(a_w, w)| \leq M |u(v) - a_w| \leq M |u(v) - l(v)|.
\end{align}

\paragraph{Case 2.} Suppose that $a_w \leq l(v) \leq u(v)$. Then, as $a \mapsto F(a, v)$ is non-decreasing on $[a_w, \infty)$,
\begin{align}
\Lambda(u,v) = F(l(v), w) \leq F(a(v), w) \leq F(u(v), w) = \Gamma(u,v),
\end{align}
and $|\Gamma(u,v) - \Lambda(u,v)| \leq M |u(v) - l(v)|$.

\paragraph{Case 3.} Suppose that $l(v) \leq u(v) \leq a_w$. Then, as $a \mapsto F(a, v)$ is non-increasing on $(-\infty, a_w]$,
\begin{align}
\Lambda(u,v) = F(u(v), w) \leq F(a(v), w) \leq F(l(v), w) = \Gamma(u,v),
\end{align}
and $|\Gamma(u,v) - \Lambda(u,v)| \leq M |u(v) - l(v)|$.

We have thus shown that, for all $(v,w) \in \mathcal{V} \times \mathcal{W}$,
\begin{align}
\Lambda(v, w) \leq F(a(v), w) \leq \Gamma(v, w),
\end{align}
and 
\begin{align}
\Gamma(v, w) - \Lambda(v,w)| \leq M |u(v) - l(v)|.
\end{align}
By integration of the above display, we have that
\begin{align}
\| \Gamma - \Lambda \|_{P,r} \leq M \|u - l \|_{P,r}.
\end{align}
Therefore, we have shown that an $(\epsilon, \| \cdot \|_{P,r})$-bracket for $\mathcal{A}$ induces an $(M \epsilon, \|\cdot\|_{P,r})$-bracket for $\mathcal{B}$.
Therefore, for all $\epsilon > 0$,
\begin{align}
N_{[]}(\epsilon, \mathcal{B}, \|\cdot\|_{P,r}) \leq N_{[]} (\epsilon / M, \mathcal{A}, \| \cdot\|_{P,r}),
\end{align}
and, for all $\delta > 0$,
\begin{align}
J_{[]}(\delta, \mathcal{B}, \|\cdot\|_{P,r}) &\leq \int_0^\delta  \sqrt{\log N_{[]}(\epsilon / M, \mathcal{A}, \|\cdot\|_{P,r})} d \epsilon \\
&\leq M  \int_0^{\delta / M} \sqrt{\log N_{[]}(\zeta, \mathcal{A}, \|\cdot\|_{P,r})} d \zeta\\
& = M J_{[]} (\delta / M, \mathcal{A}, \|\cdot\|_{P,r}).
\end{align}
\end{proof}

\begin{proof}[Proof of lemma \ref{lemma:sup_norm_dominated_by_svn}]
Let $x \in [0,1]^d$. From the representation formula in proposition \ref{proposition:representation},
\begin{align}
f(x) = f(0) + \sum_{\emptyset \neq s \subseteq [d]} \int_{[0_s, x_s]} f(dx_s).
\end{align}
Therefore,
\begin{align}
|f(x)| &\leq |f(0)| + \sum_{\emptyset \neq s \subseteq [d]} \int_{[0_s, x_s]} |f(dx_s)| \\
&\leq |f(0)| + \sum_{\emptyset \neq s \subseteq [d]} \int_{[0_s, 1_s]} |f(dx_s)| \\
&= \|f\|_v.
\end{align}
By taking the sup with respect to $x$, we obtain the wished result.
\end{proof}

\section{Proof of propositions of section \ref{section:applications}}

\subsection{Proof of results on least-squares with bounded dependent variable}

\begin{proof}[Proof of proposition \ref{proposition:least-squares_Lipschitz_loss}]
Let $y \in [-\tilde{a}_n, \tilde{a}_n]$. It is clear that $u\mapsto \tilde{L}(u,y) = (y-u)^2$ is non-increasing on $(-\infty, y]$ and non-decreasing on $[y, \infty)$. 

We now turn to showing the Lipschitz property claim. Observe that
\begin{align}
\mathcal{U}_n \equiv \{\theta(x) : \theta \in \Theta_n, x \in [0,1]^d \} \subseteq [-\tilde{a}_n, \tilde{a}_n].
\end{align}
Let $u_1, u_2 \in \mathcal{U}_n$. We have that
\begin{align}
| \tilde{L}(u_1, y) - \tilde{L}(u_2,y)| =& |(y-u_2)^2 - (y-u_1)^2| \\
=& |2 y - u_1 - u_2 | |u_2 - u_1| \\
\leq & 4 \tilde{a}_n |u_2 - u_1|,
\end{align}
which is the wished claim.
\end{proof}

The proof of proposition \ref{proposition:least-squares_smooth_loss} requires the following lemma.

\begin{lemma}\label{lemma:ls_lbd_dominates_l2_norm}
Consider $\Theta_n$, $\theta_n$, $\theta_0$, and $d_n$ as defined in subsection \ref{subsection:ls_with_bounded_y}. Then, for all $\theta \in \Theta$,
\begin{align}
d_n^2(\theta, \theta_n) = \|\theta - \theta_0 \|_{P_0,2}^2 - \|\theta - \theta_n\|_{P_0,2}^2 \geq \| \theta - \theta_0 \|_{P_0,2}^2.
\end{align}
\end{lemma}

\begin{proof}
It is straighforward to check that $\Theta_n$ is a closed convex set. Denote, for all $\theta_1, \theta_2$, $\langle \theta_1, \theta_2 \rangle = E_{P_0} [\theta_1 (X) \theta_2(X) ]$. Observe that, for all $\theta \in \Theta_n$, $\|\theta\|_{P_0,2}^2 \langle \theta, \theta \rangle$. Let $\theta \in \Theta_n$. We have that
\begin{align}
d_n^2(\theta, \theta_n) =& E_{P_0} [ (Y-  \theta(X))^2 ] - E_{P_0} [(Y- \theta_n(X))^2] \\
= & E_{P_0} [(Y- \theta_0(X))^2 ] + E_{P_0}[(\theta_0(X) - \theta(X))^2] \\
&- \left\lbrace E_{P_0} [(Y- \theta_0(X))^2 ] + E_{P_0}[(\theta_0(X) - \theta_n(X))^2] \right\rbrace\\
=& \|\theta - \theta_0\|_{P_0,2}^2 - \|\theta_n - \theta_0 \|_{P_0,2}^2.
\end{align}
Therefore,
\begin{align}
d_n^2(\theta, \theta_n) - \|\theta - \theta_n \|_{P_0,2}^2 =& \| (\theta - \theta_n) + (\theta - \theta_n) \|^2_{P_0,2} - \|\theta_n - \theta_0\|^2_{P_0,2} - \| \theta - \theta_0\|_{P_0,2}^2\\
=& - 2 \langle \theta - \theta_n, \theta_0 - \theta_n \rangle \\
\geq & 0.
\end{align}
The last line follows from the fact that $\theta \in \Theta_n$ and that $\theta_n$ is the projection for the $\|\cdot\|_{P_0,2}$ of $\theta_0$ onto the closed convex set $\Theta_n$.
\end{proof}

We can now state the proof of proposition \ref{proposition:least-squares_smooth_loss}.

\begin{proof}[Proof of proposition \ref{proposition:least-squares_smooth_loss}]
From proposition \ref{proposition:least-squares_Lipschitz_loss}, for all $o = (x,y) \in [0,1]^d \times [-\tilde{a}_n, \tilde{a}_n ]$, $|L(\theta)(o) - L(\theta)(o) |\leq |\theta(x) - \theta_n(x) |.$ Therefore, by integration
\begin{align}
\|L(\theta) - L(\theta_n) \|_{P_0,2} \leq & 4 \tilde{a}_n \|\theta - \theta_n \|_{P_0,2} \\
\leq & 4 \tilde{a}_n d_n(\theta, \theta_n),
\end{align}
where the last line follows from lemma \ref{lemma:ls_lbd_dominates_l2_norm}.
\end{proof}

\subsection{Proofs of the results on logistic regression}

\begin{proof}[Proof of proposition \ref{proposition:logistic_regression-unimodal-Lipschitz_loss}]
Let $y \in \{0,1\}$. It is clear that $u \mapsto \tilde{L}(u, y)$ is non-increasing on $\mathbb{R}$ if $y = 1$ and non-decreasing on $\mathbb{R}$ if $y = 0$. Let's now turn to the Lipschitz property claim. For all $u \in \mathbb{R}$, 
\begin{align}
\frac{\partial \tilde{L}}{\partial u}(u,y) = \frac{1}{1 + e^{-u}} - y.
\end{align}
Therefore, for all $u \in \mathbb{R}$, $y \in \{0,1\}$,
\begin{align}
\bigg| \frac{\partial \tilde{L}}{\partial u}(u,y) \bigg| \leq 1,
\end{align}
which implies that $\tilde{L}$ is $1$-Lipschitz in its first argument.
\end{proof}

\begin{proof}[Proof of proposition \ref{proposition:logistic_regression-smooth_loss}]
For all $x$, denote $\eta_0(x) = E_{P_0}[Y | X = x] = (1 + \exp(-\theta_0(x) )^{-1}$, and $\eta_n(x) = (1 + \exp(-\theta_n(x) )^{-1}$. For all $p \in [0,1]$, $q \in \mathbb{R}$, denote
\begin{align}
f_p(q) = p \log (1 + e^{-q}) + (1 - p) \log ( 1 +e^{-q}).
\end{align}
Observe that, for all $\theta$,
\begin{align}
P_0 L(\theta) = E_{P_0} [ f_{\eta_0(X)}(\theta(X)) ]. \label{eq:rel_risk_and_f_pq}
\end{align}
For all $p \in [0,1]$, $q \in \mathbb{R}$, we have that
\begin{align}
f_p'(q) = & \frac{1}{1 + e^{-q}} - p,\\
\text{ and } f_p''(q) = & \frac{1}{1 + e^{-q}} \times \left( 1 - \frac{1}{1 + e^{-q}} \right) \\
\geq & \frac{1}{2} \times \min \left( \frac{1}{1 + e^{-q} }, \frac{1}{1 + e^q} \right).
\end{align}
Therefore, for $p \in [0,1]$ and $q \in [\tilde{a}_n, \tilde{a}_n]$, we have that $f_p''(q) \geq 2^{-1} (1 + e^{\tilde{a}_n} )^{-1}$. 
From the above display, we have that, for all $x \in [0,1]^d$,
\begin{align}
f_{\eta_0(x)}(\theta(x) - f_{\eta_0(x)}(\theta_n(x)) \geq & f_{\eta_0(x)}'(\theta_n(x)) (\theta(x) -\theta_n(x)) + \frac{1}{4 (1 + e^{\tilde{a}_n} )} (\theta(x) - \theta_n(x))^2\\
= & (\eta_n(x) - \eta_0(x)) (\theta(x) - \theta_n(x)) + \frac{1}{4 (1 + e^{\tilde{a}_n} )} (\theta(x) - \theta_n(x))^2.
\end{align}
Therefore, for any $\theta \in \Theta_n$, using \eqref{eq:rel_risk_and_f_pq},
\begin{align}
d_n^2(\theta, \theta_n) =& P_0 L(\theta) - P_0 L(\theta_n) \\
\geq & E_{P_0}[ (\eta_n(X) - \eta_0(X)) (\theta(X) - \theta_n(X)) ] + \frac{1}{4 ( 1 +e^{\tilde{a}_n})}  \|\theta - \theta_n \|_{P_0,2}^2. \label{eq:log_reg_lower_bound_lbd}
\end{align}\
Let $\theta \in \Theta_n$.
For all $t$, define $\tilde{\theta}(t) = \theta_n + t (\theta - \theta_n)$ and $g(t) = P_0 L(\tilde{\theta}(t))$. Since $\theta_n$ and $\theta$ are in $\Theta_n$ and that $\Theta_n$ is convex, for all $t \in [0,1]$, $\tilde{\theta}(t) \in \Theta_n$. Therefore, by definition of $\theta_n$, for all $t \in [0,1]$, $g(t) \geq g(0)$. Thus, by taking the limit of $(g(t) - g(0)) / t$ as $t \downarrow 0$, we obtain that $g'(0) \geq 0$. We now calculate $g'(0)$:
\begin{align}
g'(0) =& \frac{d}{dt} \bigg\{ E_{P_0}\big[\eta_0(X) \log ( 1 + e^{-(\theta_n(X) + t(\theta(X) - \theta_n(X))} \\
&\qquad \qquad + (1 - \eta_0(X)) \log (1 + e^{\theta_n(X) + t(\theta(X)- \theta_n(X))} ) \big] \bigg\}\bigg|_{t=0}\\
=& E_{P_0} \big[ -\eta_0(X) \frac{e^{-\theta_n(X)}}{1 + e^{-\theta_n(X)} } (\theta(X) - \theta_n(X)) \\
&\qquad + ( 1 -\eta_0(X)) \frac{e^{\theta_n(X)}}{1 + e^{\theta_n(X)}} (\theta(X) - \theta_n(X)) \big] \\
=& E_{P_0} [ \{ -\eta_0(X) (1 - \eta_n(X)) + (1 -\eta_0(X)) \eta_n(X) \} (\theta(X) - \theta_n(X))] \\
=& E_{P_0} [ (\eta_n(X) - \eta_0(X)) (\theta(X) - \theta_n(X) ],
\end{align}
which is equal to the first term in the right-hand side of \eqref{eq:log_reg_lower_bound_lbd}. Therefore, as $g'(0) \geq 0$,
\begin{align}
d_n^2(\theta, \theta_n) \geq \frac{1}{4 (1 + e^{\tilde{a}_n})} \|\theta - \theta_n \|_{P_0,2}^2.
\end{align}
From proposition\ref{proposition:logistic_regression-unimodal-Lipschitz_loss}, for all $o=(x,y) \in [0,1]^d \times \{0,1\}$, $|L(\theta)(o) - L(\theta_n)(o)| \leq |\theta(x)  - \theta_n(x)|$, therefore, by integration, 
\begin{align}
\|L(\theta) - L(\theta_n)\|_{P_0,2} \leq \|\theta - \theta_n\|_{P_0,2} \leq 2 (1 + e^{\tilde{a}_n})^{-1/2} d_n(\theta, \theta_n).
\end{align}
\end{proof}
\section{Proof of the rate theorem for least-squares regression with sub-exponential errors}

We first give an informal overview of the proof. We will proceed very similarly as in the case of the proof of the rate theorem under bounded losses, that is we will first identify $\mathbb{M}_n$, $M_n$, $d_n$ that satisfy the hypothesis of theorem \ref{thm_peeling}, and then we will bound the modulus of continuity of $\mathbb{M}_n - M_n$. 

Observe that
\begin{align}
\hat{\theta}_n &= \arg \min_{\theta \in \Theta_n} \frac{1}{n} \sum_{i=1}^n (Y_i - \theta(X_i))^2 \\
&= \arg \min_{\theta \in \Theta_n}  \frac{1}{n} \sum_{i=1}^n (\theta_0(X_i) - \theta(X_i) + e_i)^2 \\
&= \arg \max_{\theta \in \Theta_n}  \frac{1}{n} \sum_{i=1}^n 2 (\theta(X_i) - \theta_0(X_i)) e_i - (\theta(X_i) - \theta_0(X_i))^2.
\end{align}
This motivates setting 
\begin{align}
    \mathbb{M}_n(\theta) = \frac{1}{n} \sum_{i=1}^n 2 (\theta - \theta_0)(X_i) e_i - (\theta -  \theta_0)(X_i),
\end{align}
and, since $E_{P_0}[(\theta - \theta_0)(X_i) e_i] = 0$, 
\begin{align}
    M_n(\theta) = -P_0 (\theta - \theta_0)^2,
\end{align}
and introducing the loss-based dissimilarity $d_n$, defined, for all $\theta \in \Theta_n$, by
\begin{align}
    d_n^2(\theta, \theta_n) = - (M_n(\theta) - M_n(\theta_n))^2.
\end{align}
The main effort will then be to upper bound, for any $\delta > 0$, the quantity 
\begin{align}
    E_{P_0} \sup_{\substack{\theta \in \Theta_n \\ d_n(\theta, \theta_n) \leq \delta}} | (\mathbb{M}_n - M_n)(\theta) - (\mathbb{M}_n - M_n)(\theta_n)|.
\end{align}

The proof relies on the following lemmas, whose proofs we defer to subsection \ref{subsection-subexp_ls-proof_technical_lemmas}.

\begin{lemma}\label{lemma-l2_norm_and_loss_based_dissimilarity} For all $\theta \in \Theta_n$,
\begin{align}
    \| \theta - \theta_n \|_{P_0, 2} \leq d_n(\theta, \theta_n).
\end{align}
\end{lemma}

For any $\theta \in \Theta_n$, we introduce the functions $g_{1,n}(\theta)$ and $g_{2,n}(\theta)$, where, for all $(x, e)$
\begin{align}
    g_{1,n}(\theta)(x,e) = (\theta(x) - \theta_n(x)) e,
\end{align}
and
\begin{align}
    g_{2,n}(\theta) = (\theta - \theta_n)(2 \theta - \theta_n - \theta_0).
\end{align}
We will consider the following two sets:
\begin{align}
    \mathcal{G}_{1,n} &= \{ g_{1,n}(\theta) : \theta \in \Theta_n \}\\
    \mathcal{G}_{2,n} &= \{ g_{2,n}(\theta) : \theta \in \Theta_n \}.
\end{align}
We will use the following version of the so-called Bernstein norm, defined for any $ t> 0$ and for any function $g:(x,e) \mapsto g(x, e)$ as 
\begin{align}
    \|g\|^2_{P_0, B, t} = t^{-2} P \phi(t  g),
\end{align}
where $\phi(x) = e^x - x -1$.
As for all $i$, $e_i$ is sub-exponential with parameters $(\alpha, \nu)$, $|e_i|$ is sub-exponential with parameters $(\alpha'(\alpha, \nu), \nu'(\alpha, \nu))$. We will shorten notations by denoting $\alpha'=\alpha'(\alpha, \nu)$ and $\nu'=\nu'(\alpha, \nu)$.
The following lemma characterizes the Bernstein norm of a certain type of functions.
\begin{lemma}\label{lemma-Bernstein_norm_f_times_e}
Let $f: \mathcal{X} \rightarrow \mathbb{R}$ such that $\|f\|_\infty \leq M$. Suppose that $M \geq 1$. Consider $g_1:(x, e) \mapsto f(x) e$. 
Then, setting $t = (\alpha M)^{-1}$, we have
\begin{align}
    \|g_1\|_{P_0, B, t} \leq \|f\|_{P_0, 2} \alpha M e^{\nu^2 / (4 \alpha ^2) }.
\end{align}
Similarly, now consider $g_2:(x, e) \mapsto f(x) |e|$. Setting $t = (\alpha' M)^{-1}$, we have
\begin{align}
    \|g_2\|_{P_0, B, t} \leq \|f\|_{P_0, 2} \alpha' M e^{\nu'^2 / (4 \alpha'^2) }.
\end{align}
\end{lemma}
This has the following immediate corollary for $g_{1,n}$. In this following result as well as in the rest of this section, we will denote $t_n = (2_a \alpha')^{-1}$.
\begin{corrolary}\label{corollary-g1n}
We have that for all $\theta \in \Theta_n$,
\begin{align}
    \| g_{1,n}(\theta) \|_{P_0, B, t_n} \leq C_n \|\theta - \theta_n\|_{P_0,2},
\end{align}
where $C_n = \tilde{C}(\alpha, \nu) a_n$, with $\tilde{C}(\alpha, \nu) = 2 \alpha'(\alpha, \nu) e^{\nu'(\alpha, \nu)^2 / (4 \alpha'(\alpha, \nu)^2)}.$
\end{corrolary}

The upcoming lemma relates the bracketing numbers in $\|\cdot\|_{P_0, B, t_n}$ norm of $\mathcal{G}_{1,n}$ to the bracketing numbers of $\Theta_n$ in $\|\cdot\|_{P_0,2}$ norm.
\begin{lemma}\label{lemma-G1n_characterization}
For any $\epsilon > 0$,
\begin{align}
    N_{[]}(\epsilon, \mathcal{G}_{1,n}, \|\cdot\|_{P_0, B, t_n} ) \leq N_{[]}(C_n^{-1} \epsilon, \Theta_n, \|\cdot\|_{P_0, 2} ),
\end{align}
and the bracketing entropy integral of $\mathcal{G}_{1,n}$ satisfies, for all $\delta > 0$,
\begin{align}
    J_{[]}(\delta, \mathcal{G}_{1,n}, \|\cdot\|_{P_0, B, t_n}) \leq C_n J_{[]}(C_n^{-1} \delta, \Theta_n, \|\cdot\|_{P_0, 2} ).
\end{align}
\end{lemma}

The upcoming lemma relates characterizes the $\|\cdot\|_{P_0,2}$ and the $\|\cdot\|_\infty$ norm of $g_{2,n}$ and the bracketing numbers in $\|\cdot\|_{P_0,2}$ norm of $\mathcal{G}_{2,n}$.
\begin{lemma}\label{lemma-G2n_characterization}
Consider $g_{2,n}$ defined above. For every $\theta \in \Theta_n$,
\begin{align}
    \|g_{2,n}(\theta)\|_{P_0, 2} &\leq (\|\theta_0\|_\infty + 3 a_n) \|\theta - \theta_n \|_{P_0,2}\\
    \| g_{2,n}(\theta) \|_\infty &\leq 2 a_n (\|\theta_0 \|_\infty + 3 a_n),
\end{align}
and, for all $\epsilon > 0$,
\begin{align}
    N_{[]}(\epsilon, \mathcal{G}_{2,n}, \|\cdot\|_{P_0, 2}) \leq N_{[]}( (\|\theta_0\|_\infty + 3 a_n)^{-1} \epsilon, \Theta_n, \|\cdot\|_{P_0,2} ),
\end{align}
and, for all $\delta > 0$,
\begin{align}
    J_{[]}(\delta, \mathcal{G}_{2,n}, \|\cdot\|_{P_0,2}) \leq (\|\theta_0\|_\infty + 3 a_n ) J_{[]}( (\| \theta_0 \|_\infty^{-1} + 3 a_n)^{-1} \delta, \Theta_n, \|\cdot\|_{P_0,2} ).
\end{align}
\end{lemma}

In addition to lemma \ref{lemma-maximal_inequality_L2_brackets} (lemma 3.4.2 from \cite{vdV-Wellner-1996}), we will use the maximal inequality of lemma 3.4.3 from \cite{vdV-Wellner-1996}, which we restate here.

\begin{lemma}[Lemma 3.4.3 in \cite{vdV-Wellner-1996}]\label{lemma-maximal_inequality_LPB_brackets} Let $\mathcal{F}$ be a class of measurable functions such that $\|f\|_{P, B} \leq \delta$ for every $f \in \mathcal{F}.$ Then 
\begin{align}
    E_P^* \sup_{f\in \mathcal{F}} |\sqrt{n}(P_n - P_0) f| \leq J_{[]}(\delta, \mathcal{F}, \|\cdot\|_{P,B}) \left( 1 + \frac{J_{[]}(\delta, \mathcal{F}, \|\cdot\|_{P,B}) }{\delta^2 \sqrt{n}} \right).
\end{align}

\end{lemma}

We can now present the proof of theorem \ref{thm-rate-least_squares_with_sub-exponential_errors}.
\begin{proof}
We will oragnize the proof in three steps

\paragraph{Step 1: Checking that $\mathbb{M}_n$, $M_n$, and $d_n$ satisfy the conditions of theorem \ref{thm_peeling}.}
\begin{itemize}
    \item By definition of $d_n$, for all $\theta \in \Theta_n$, $M_n(\theta) - M_n(\theta_n) = - d^2_n(\theta, \theta_n)$, therefore condition \ref{thm_peeling-population_risk_and_distance_condition} is satisfied. 
    \item By definition of $\hat{\theta}_n$, $\mathbb{M}_n(\hat{\theta}_n) \geq \mathbb{M}_n(\theta_n) - O_P(r_n^{-2})$.
\end{itemize}
We will apply the theorem with $\eta = \infty$.

\paragraph{Step 2: Bounding the modulus of continuity.}
We have that
\begin{align}
    & E_{P_0} \sup_{ \substack{\theta \in \Theta_n \\ d_n(\theta, \theta_n) \leq \delta} } | (\mathbb{M}_n - M_n)(\theta) - (\mathbb{M}_n - M_n)(\theta_n)| \\
    =& E_{P_0} \sup_{ \substack{\theta \in \Theta_n \\ d_n(\theta, \theta_n) \leq \delta} } \bigg| \frac{2}{n} \sum_{i=1}^n (\theta(X_i) - \theta_n(X_i)) e_i + (P_n - P_0) ((\theta - \theta_0)^2 - (\theta_n - \theta_0)^2) \bigg| \\
    =& 2 E_{P_0} \sup_{ \substack{\theta \in \Theta_n \\ d_n(\theta, \theta_n) \leq \delta} } |(P_n - P_0) g_{1,n}(\theta)| +  E_{P_0} \sup_{ \substack{\theta \in \Theta_n \\ d_n(\theta, \theta_n) \leq \delta} } |(P_n - P_0) g_{2,n}(\theta)|,
\end{align}
with $g_{1,n}$ and $g_{2,n}$ as defined above. From lemma \ref{lemma-l2_norm_and_loss_based_dissimilarity}, for any $\theta \in \Theta_n$, $\|\theta - \theta_n\|_{P_0,2} \leq d(\theta, \theta_n)$, and from corollary \ref{corollary-g1n} and lemma \ref{lemma-G2n_characterization}, that $\|\theta - \theta_n\|_{P_0, 2} \leq \delta$ implies that $\|g_{1,n}(\theta)\|_{P_0, B, t_n} \leq C_n \delta$ and $(\|\theta_0 \|_\infty + 3 a_n) \delta$. Therefore, the right-hand side of the above display is upper-bounded by 
\begin{align}
    2 E_{P_0} \sup_{\substack{g \in \mathcal{G}_{1,n} \\ \|g\|_{P_0, B, t_n} \leq C_n \delta}} |(P_n - P_0) g| + E_{P_0} \sup_{\substack{g \in \mathcal{G}_{2,n} \\ \|g\|_{P_0, 2} \leq (\|\theta_0\|_{P_0, 2} + 3 a_n) \delta}} |(P_n - P_0) g|,
\end{align}
where $\mathcal{G}_{1,n}$ and $\mathcal{G}_{2,n}$ are as defined above.

From lemma \ref{lemma-maximal_inequality_L2_brackets} and lemma \ref{lemma-maximal_inequality_LPB_brackets}, we can bound the above display by
\begin{align}
    & J_{[]}(C_n \delta, \mathcal{G}_{1,n}, \|\cdot\|_{P, B, t_n}) \left( 1 + \frac{J_{[]}(C_n \delta, \mathcal{G}_{1,n}, \|\cdot\|_{P, B, t_n}) }{C_n^2 \delta^2 \sqrt{n}} \right) \\
    &+ J_{[]}((\|\theta_0\|_\infty + 3 a_n) \delta, \mathcal{G}_{2,n}, \|\cdot\|_{P_0,2}) \left(1 + \frac{J_{[]}((\|\theta_0\|_\infty + 3 a_n) \delta, \mathcal{G}_{2,n}, \|\cdot\|_{P_0,2} ) 2a_n (\|\theta_0\|_\infty + 3 a_n) }{ (\|\theta_0\|_\infty + 3 a_n)^2 \delta^2 \sqrt{n}} \right) \\
    \leq  & ( J_{[]}(C_n \delta, \mathcal{G}_{1,n}, \|\cdot\|_{P, B, t_n}) +  J_{[]}((\|\theta_0\|_\infty + 3 a_n) \delta, \mathcal{G}_{2,n}, \|\cdot\|_{P_0,2}) ) \\
    & \times \left( 1  + \frac{J_{[]}(C_n \delta, \mathcal{G}_{1,n}, \|\cdot\|_{P, B, t_n}) }{C_n^2 \delta^2 \sqrt{n}}
 + \frac{J_{[]}((\|\theta_0\|_\infty + 3 a_n) \delta, \mathcal{G}_{2,n}, \|\cdot\|_{P_0,2} ) 2a_n (\|\theta_0\|_\infty + 3 a_n) }{ (\|\theta_0\|_\infty + 3 a_n)^2 \delta^2 \sqrt{n}} \right) \label{proof_thm_subexp_ls-max_ineqs}.
\end{align}
From lemma \ref{lemma-G1n_characterization},
\begin{align}
    J_{[]}(C_n \delta, \mathcal{G}_{1,n}, \|\cdot\|_{P, B, t_n}) \leq C_n J_{[]}(\delta, \Theta_n, \|\cdot\|_{P_0,2}).
\end{align}
Therefore,
\begin{align}
    \frac{ J_{[]}(C_n \delta, \mathcal{G}_{1,n}, \|\cdot\|_{P, B, t_n}) }{C_n^2 \delta^2 \sqrt{n}} \leq C_n^{-1} \frac{J_{[]}(\delta, \Theta_n, \|\cdot\|_{P_0,2})} {\delta^2 \sqrt{n}}.
\end{align}
From lemma \ref{lemma-G2n_characterization},
\begin{align}
    J_{[]}((\|\theta_0\|_\infty + 3 a_n) \delta, \mathcal{G}_{2,n}, \|\cdot\|_{P_0,2}) \leq (\|\theta_0\|_{P_0,2} + 3 a_n) J_{[]}(\delta, \Theta_n, \|\cdot\|_{P_0,2} ).
\end{align}
Therefore,
\begin{align}
  \frac{ J_{[]}((\|\theta_0\|_\infty + 3 a_n) \delta, \mathcal{G}_{2,n}, \|\cdot\|_{P_0,2})  2a_n (\|\theta_0\|_\infty + 3 a_n) }{(\|\theta_0\|_\infty + 3 a_n)^2 \delta^2 \sqrt{n}}  \leq (\|\theta_0\|_\infty + 3 a_n) J_{[]}(\delta, \Theta_n, \|\cdot\|_{P_0,2} ).
\end{align}
Therefore, we can bound \eqref{proof_thm_subexp_ls-max_ineqs} by 
\begin{align}
    &(C_n + 3 a_n + \|\theta_0\|_\infty) J_{[]}(\delta, \Theta_n, \|\cdot\|_{P_0,2} ) \left(1 + \frac{ (C_n^{-1} + 3 a_n + \|\theta_0\|_\infty) J_{[]}(\delta, \Theta_n, \|\cdot\|_{P_0,2} )}{\delta^2 \sqrt{n}} \right) \\
    & \lesssim  \phi_n(\delta),
\end{align}
with
\begin{align}
    \phi_n(\delta) \equiv ((\tilde{C}(\alpha, \nu) + 3) a_n + \|\theta_0\|_\infty) J_{[]}(\delta, \Theta_n, \|\cdot\|_{P_0,2} ) \left(1 + \frac{ (C_n^{-1} + 3 a_n + \|\theta_0\|_\infty) J_{[]}(\delta, \Theta_n, \|\cdot\|_{P_0,2} )}{\delta^2 \sqrt{n}} \right).
\end{align}

\paragraph{Step 3: Checking the rate condition.}
Recall that we set 
$$ r_n = C(r,d)^{-1/3}((\tilde{C} + 3 ) a_n + \|\theta_0\|_\infty )^{-1} (\log n)^{-2(d-1)/3} n^{1/3}. $$
Therefore, 
\begin{align}
    &r_n^2 ((\tilde{C} + 3 ) a_n + \|\theta_0\|_\infty ) J_{[]}(r_n^{-1}, \Theta_n, \|\cdot\|_{P_0,2}) \\
    &\lesssim r_n^2 ((\tilde{C} + 3 ) a_n + \|\theta_0\|_\infty ) C(r,d)^{1/2} a_n^{1/2} r_n^{-1/2} (\log (a_n r_n))^{d-1} \\
    &\lesssim r_n^{3/2} ((\tilde{C} + 3 ) a_n + \|\theta_0\|_\infty )^{3/2} C(r,d)^{1/2} (\log n)^{d-1} \\
    &\lesssim \sqrt{n},
\end{align}
where, we used in the third line above, that since $a_n = O(n^p)$ for some $p> 0$, $\log (a_n r_n) = O(\log n)$, and in the fourth line, we replaced $r_n$ with its expression.
Therefore,
\begin{align}
    r_n^2 \phi(1 / r_n) \lesssim \sqrt{n},
\end{align}
which concludes the proof.
\end{proof}

\subsection{Proofs of the technical lemmas}\label{subsection-subexp_ls-proof_technical_lemmas}

\begin{proof}[Proof of lemma \ref{lemma-l2_norm_and_loss_based_dissimilarity}]
The proof follows easily from observing that $\Theta_n$ is convex and that $\theta_n$ is the projection on $\Theta_n$ of $\theta_0$ for the $\|\cdot\|_{P_0,2}$ norm.
\end{proof}

\begin{proof}[Proof of lemma \ref{lemma-Bernstein_norm_f_times_e}]
By definition of the Bernstein norm, and using the power series expansion of $\phi$, we have
\begin{align}
    \|g_1\|_{P_0, B, t}^2 =& t^{-2} \sum_{k=2}^\infty t^k \frac{P_0 (f^k e^k)}{k!} \\
    =& t^-2 \sum_{k=2}^\infty t^k \frac{P_0 (f^k) P_0 e^k}{k!} \\
    \leq & t^{-2} \|f\|_{P_0,2}^2 \sum_{k=2}^\infty \frac{t^k M^{k-2} E_{P_0}[e^k]}{k!} \\
    \leq & t^{-2} \|f\|_{P_0,2}^2 \sum_{k=2}^\infty \frac{t^k M^k E_{P_0}[e^k]}{k!} \\
    \leq & t^{-2} \|f\|_{P_0,2}^2 E_{P_0} [e^{t M e}] \\
    \leq & t^{-2} \|f\|_{P_0, 2}^2 e^{\frac{\nu^2}{2 \alpha^2}}.
\end{align}
The second line in the above display follows from the fact that $X$ and $e$ are independent under $P_0$. The fourth line uses that $M \geq 1$, which implies that $M^{k-2} \leq M^k$. The sixth line uses that $e$ is sub-exponential with parameters $(\alpha, \nu)$.
This proves the first claim.

The second claim follows by the exact same reasoning, by replacing $e$ with $|e|$ in the above developments and using that for $t=(\alpha' M)^{-1}$, $E_{P_0}[e^{t M |e|} ]  \leq e^{\frac{\nu'^2}{2 \alpha'^2}}$.
\end{proof}

\begin{proof}[Proof of lemma \ref{lemma-G1n_characterization}]. Let $\theta \in \Theta_n$
Consider $[l, u]$ an $(\epsilon, \|\cdot\|_{P_0,2})$-bracket for $\theta$. By appropriately thresholding $l$ and $u$, we can ensure that $l, u$ have values in $[-a_n, a_n]$ while still preserving that $l \leq \theta \leq u$ and $\|l - u\|_{P_0,2} \leq \epsilon$. For all $x, e$, we have that
\begin{align}
    \Lambda(x,e) \leq (\theta(x) - \theta_n(x)) e \leq \Gamma(x,e),
\end{align}
where 
\begin{align}
    \Lambda(x, e) = (l-\theta_n)(x) e^+ + (u - \theta_n)(x) e^-,
\end{align}
and 
\begin{align}
    \Gamma(x, e) = (u-\theta_n)(x) e^+ + (l - \theta_n)(x))e^-.
\end{align}
For all $x, e$, 
\begin{align}
    \Gamma(x, e) - \Lambda(x,e) = (u-l)(x) |e|.
\end{align}
Set $t_n = (2 a_n \alpha')$. From lemma \ref{lemma-Bernstein_norm_f_times_e}, $\| \Gamma - \Lambda \|_{P_0, B, t_n} \leq 2 \alpha' M e^{\nu'(\alpha, \nu)^2 / (4 \alpha'(\alpha, \nu)^2)} \epsilon$.

We have just shown that an $(\epsilon, \|\cdot\|_{P_0,2})$-bracketing of $\Theta_n$ induces a $(C_n \epsilon,\| \cdot \|_{P_0, B, t_n})$-bracketing of $\mathcal{G}_{1,n}$, which implies that
\begin{align}
    N_{[]}(\epsilon, \mathcal{G}_{1,n}, \|\cdot\|_{P_0, B, t_n}) \leq N_{[]}(C_n^{-1} \epsilon, \Theta_n, \|\cdot\|_{P_0,2}).
\end{align}
Therefore, using the above bound on the bracketing number of $\mathcal{G}_{1,n}$, and doing a change of variable in the integral, we obtain that
\begin{align}
    J_{[]}(\delta, \mathcal{G}_{1,n}, \|\cdot\|_{P_0, B, t_n})  = & \int_0^\delta \sqrt{ \log N_{[]}(\epsilon, \mathcal{G}_{1,n}, \|\cdot\|_{P_0, B, t_n}) } d \epsilon \\
    \leq & \int_0^\delta \sqrt{ \log N_{[]}(C_n^{-1} \epsilon, \Theta_n, \|\cdot\|_{P_0,2}) } d \epsilon \\
    \leq & C_n \int_0^{C_n^{-1} \delta} \sqrt{ \log N_{[]}(u, \Theta_n, \|\cdot\|_{P_0,2}) } d u \\
    \leq & C_n J_{[]}(C_n^{-1} \delta, \Theta_n, \|\cdot\|_{P_0,2}).
\end{align}
\end{proof}

\begin{proof}[Proof of lemma \ref{lemma-G2n_characterization}]
The first two claims are elementary.

We turn to the claim on the bracketing numbers. Let $[l,u]$ be an $(\epsilon, \|\cdot\|_{P_0,2})$-bracketing of $\Theta_n$. Defining
\begin{align}
    \Lambda &= u (2 \theta - \theta_0 - \theta_n)^+ - l (2 \theta - \theta_0 - \theta_n)^-\\
    \text{and } \Gamma &= l (2 \theta - \theta_0 - \theta_n)^+ - u (2 \theta - \theta_0 - \theta_n)^-,
\end{align}
we have that $\Lambda \leq (\theta - \theta_n) (2 \theta - \theta_0 - \theta_n) \leq \Gamma$. Observe that
\begin{align}
    \Gamma - \Lambda = (u-l) | 2 \theta - \theta_0 -\theta_n|.
\end{align}
Therefore $\|\Gamma - \Lambda\|_{P_0,2} \leq \epsilon (3a_n + \|\theta_0\|_\infty)$. This proves that an $(\epsilon, \|\cdot\|_{P_0,2})$-bracketing of $\Theta_n$ induces an $(\epsilon ( \|\theta_0\|_\infty + 3 a_n), \|\cdot\|_{P_0,2})$-bracketing of $\mathcal{G}_{2,n}$. From there, proceeding as in the proof of lemma \ref{lemma-G1n_characterization} yields the claims on the bracketing number and the bracketing entropy integral.
\end{proof}
\end{document}